\documentclass[11pt,a4paper]{article}
\usepackage{amsmath}
\usepackage{amsthm}
\usepackage{makeidx}
\usepackage{latexsym}
\usepackage{graphicx}
\usepackage{amssymb}
\usepackage{hyperref}
\usepackage[latin1]{inputenc}

\newtheorem{theorem}{Theorem}
\newtheorem{lemma}{Lemma}
\newtheorem{proposition}{Proposition}

\newtheorem{remark}{Remark}
\newtheorem{definition}{Definition}
\theoremstyle{remark}
\newtheorem{example}{\textbf{Example}}

\title{Groups and monoids of Pythagorean triples connected to conics}
\author{Marco Abrate, Stefano Barbero, Umberto Cerruti, Nadir Murru\\
Department of Mathematics \\ University of Turin \\ Via Carlo Alberto 8/10, Turin, Italy\\ marco.abrate@unito.it\\stefano.barbero@unito.it\\umberto.cerruti@unito.it\\nadir.murru@unito.it}
\date{}
\begin{document}

\maketitle

\begin{abstract}
We define operations that give the set of all Pythagorean triples a structure of commutative monoid. In particular, we define these operations by using injections between integer triples and $3 \times 3$ matrices. Firstly, we completely characterize these injections that yield commutative monoids of integer triples. Secondly, we determine commutative monoids of Pythagorean triples characterizing some Pythagorean triple preserving matrices. Moreover, this study offers unexpectedly an original connection with groups over conics. Using this connection, we determine groups composed by Pythagorean triples with the studied operations.
\end{abstract}


\section{Introduction}
Algebraic structures, like groups and monoids, are widely studied in mathematics for their importance. It is very interesting to give particular sets an alegebraic structure, in order to deepen their study. The set of Pythagorean triples and their properties have been extensively studied due to the interest that Pythagorean triples have aroused during the years, supporting mathematicians to continue their study. In spite of these studies, it is an hard work to give Pythagorean triples algebraic structures. In \cite{Tau} and \cite{Eck}, an operation over Pythagorean triples has been studied, showing a group of primitive Pythagorean triples. In \cite{Bea} and \cite{Som}, a monoid structure has been provided.

In this paper, we show an original approach that allows to characterize a family of operations which determine commutative monoids of Pythagorean triples. In section \ref{sec:mono}, we define injections, which we call \emph{natural}, between integer triples and $3 \times 3$ matrices. We use these injections to define products between integer triples, focusing on Pythagorean ones. In section \ref{sec:mono-mat}, we characterize matrices generating natural injections that determine commutative monoids over $\mathbb Z^3$. Similarly, in section \ref{sec:mono-pyt}, we characterize a subset of Pythagorean triple preserving matrices, such that the set of all Pythagorean triples has a commutative monoid structure. Finally, in section \ref{sec:group}, we find a surprisingly connection between Pythagorean triples and conics. In this way, we determine a group composed by Pythagorean triples by means of a morphism between points over conics and Pythagorean triples. Section \ref{sec:conc} is devoted to conclusions.

\section{Commutative monoids of Pythagorean triples}\label{sec:mono}
Let $\mathcal P$ be the set of Pythagorean triples
$$\mathcal P= \{  (x,y,z)\in\mathbb Z^3 : x^2+y^2=z^2 \}.$$
Several operations can be defined over this set. In particular, Taussky \cite{Tau} and Eckert \cite{Eck} studied the following operation
$$(x_1,y_1,z_1)\cdot(x_2,y_2,z_2)=(x_1x_2-y_1y_2,x_1y_2+y_1x_2,z_1z_2)$$
which defines a structure of free abelian group over primitive Pythagorean triples. In \cite{Zan}, this result has been generalized from $\mathbb Z$ to any ring of integers of an algebraic number field. In \cite{Bea} and \cite{Som}, the following operation has been studied
$$(x_1,y_1,z_1)\cdot(x_2,y_2,z_2)=(x_1x_2,y_1z_2+y_2z_1,y_1y_2+z_1z_2),$$
determining a commutative monoid (i.e., a commutative semigroup with identity) over $\mathcal P$.

Here, we study new operations on Pythagorean triples starting from an injection between triples and matrices. These operations provide new commutative monoid structures over $\mathcal P$. Let $\nu$ be an injective function
$$\nu: \mathbb Z^3 \rightarrow M_3,$$
where $M_3$ is the set of $3\times3$ matrices whose entries belong to $\mathbb Z$. Clearly, $\nu$ can be used to induce products over $\mathbb Z^3$. Let us denote with ''.'' the usual matrix product. One natural way to induce a product over $\mathbb Z^3$ is given by
$$\textbf{a}\bullet_1\textbf{b}=\nu^{-1}(\nu(\textbf{a}).\nu(\textbf{b})), \quad\forall \textbf{a},\textbf{b}\in\mathbb Z^3.$$
A second natural way is given by
$$\textbf{a}\bullet_2\textbf{b}=\nu(\textbf{a}).\textbf{b},\quad\forall \textbf{a},\textbf{b}\in\mathbb Z^3.$$
Let us consider the following definition.
\begin{definition}\label{natural}
We say that injection $\nu$ is \emph{natural} when
\begin{enumerate}
\item $\nu(\nu(\textbf{a}).\textbf{b})=\nu(\textbf{a}).\nu(\textbf{b})$, $\forall \textbf{a},\textbf{b}\in\mathbb Z^3$
\item $\nu(\textbf{a}).\textbf{b}=\nu(\textbf{b}).\textbf{a}$, $\forall \textbf{a},\textbf{b}\in\mathbb Z^3$
\item $\nu(1,0,1)=I_3$ (identity matrix)
\end{enumerate}
\end{definition}
\begin{remark}
We consider an element $\textbf{a}\in\mathbb Z^3$ like a column vector or a row vector in order that products between $\textbf{a}$ and matrices are consistent.
\end{remark}
In the following, we are interested in natural $\nu$. Thus, we will consider the induced product $\bullet=\bullet_1=\bullet_2$, i.e.,
$$\textbf{a}\bullet \textbf{b}=\nu(\textbf{a}).\textbf{b}=\nu^{-1}(\nu(\textbf{a}).\nu(\textbf{b})),\quad \forall \textbf{a},\textbf{b}\in\mathbb Z^3.$$
If $\nu$ is natural, then $(\mathbb Z^3,\bullet)$ is a commutative monoid, whose identity is $(1,0,1)$.

We also want to determine when the product $\bullet$ preserves Pythagorean triples, i.e.,
$$\forall \textbf{a},\textbf{b} \in\mathcal P, \quad \nu(\textbf{a}).\textbf{b}\in\mathcal P.$$
This study is clearly related to research Pythagorean triple preserving matrices.
\begin{definition}
A matrix $A$ is called a \emph{Pythagorean triple preserving matrix} (PTPM) if given a Pythagorean triple $\textbf{a}$, then $A.\textbf{a}$ is still a Pythagorean triple.
\end{definition}
In \cite{Pal1} and \cite{Pal2}, the authors characterized PTPMs. Precisely, they are of the form
$$\begin{pmatrix}  \frac{1}{2}(r^2-t^2-s^2+u^2) & rs-tu & \frac{1}{2}(r^2-t^2+s^2-u^2) \cr rt-su & ru+st & rt+su \cr \frac{1}{2}(r^2+t^2-s^2-u^2) & rs+tu & \frac{1}{2}(r^2+t^2+s^2+u^2) \end{pmatrix}$$
for some values of $r,s,t,u$. Further results can be found in \cite{Cra}. Moreover, Tikoo \cite{Tik} studied special cases of PTPMs provided by matrices
$$B_1=\begin{pmatrix} x & 0 & 0 \cr 0 & z & y \cr 0 & y & z \end{pmatrix} \quad B_2=\begin{pmatrix} y & 0 & 0 \cr 0 & z & x \cr 0 & x & z \end{pmatrix} \quad B_3=\begin{pmatrix} z & 0 & x \cr 0 & y & 0 \cr x & 0 & z \end{pmatrix}$$ 
$$B_4=\begin{pmatrix} z & 0 & y \cr 0 & x & 0 \cr y & 0 & z \end{pmatrix} \quad B_5=\begin{pmatrix} -x & y & 0 \cr y & x & 0 \cr 0 & 0 & z \end{pmatrix}$$
for $(x,y,z)$ a given Pythagorean triple. All these matrices induce products that preserve Pythagorean triples. In other words, we can consider, e.g., the injection
$$\nu_{B_1}:\mathbb Z^3 \rightarrow M_3,\quad \nu_{B_1}(x,y,z)=B_1.$$
Thus, the induced product $\bullet_{B_1}$ preserves Pythagorean triples, i.e.,
$$\forall \textbf{a}, \textbf{b}\in \mathcal P,\quad \textbf{a}\bullet_{B_1}\textbf{b}\in\mathcal P.$$
Moreover, with a bit of calculation, it is possible to check that $(\mathcal P,\bullet_{B_1})$ is a commutative monoid or equivalently $\nu_{B_1}$ is natural. Similarly, we can consider the injection $\nu_{B_2}$ and the product $\bullet_{B_2}$. Clearly, $\bullet_{B_2}$ preserves Pythagorean triples. However, $\bullet_{B_2}$ is not commutative nor associative and there is not the identity. In this case $(\mathcal P,\bullet_{B_2})$ is not a commutative monoid or equivalently $\nu_{B_2}$ is not natural.

Thus, it is natural asking when PTPMs determine commutative and associative products with identity. In the following we answer to this question finding a characterization for these matrices.

\section{Matrices yielding commutative monoids of triples}\label{sec:mono-mat}
In this section we determine all matrices that yield natural injections $\nu$. In other words, we determine all matrices such that $(\mathbb Z^3,\bullet)$ is a commutative monoid whose identity is (1,0,1). Using previous notation, we have that $\nu(x,y,z)\in M_3$ has entries depending on triple $(x,y,z)$.

We consider generic $3\times3$ matrices whose entries are functions of $(x,y,z)$:
$$\nu(x,y,z)=\begin{pmatrix} a(x,y,z) & b(x,y,z) & c(x,y,z) \cr d(x,y,z) & e(x,y,z) & f(x,y,z) \cr g(x,y,z) & h(x,y,z) & i(x,y,z)  \end{pmatrix}. $$
Triple $(1,0,1)$ is the identity with respect to $\bullet$ if and only if
$$(x,y,z)\bullet(1,0,1)=(1,0,1)\bullet(x,y,z)=(x,y,z),\quad \forall (x,y,z)\in\mathbb Z^3$$
i.e., if and only if
$$\nu(x,y,z).(1,0,1)=\nu(1,0,1).(x,y,z)=(x,y,z).$$
Direct calculations show that these equalities are satisfied if and only if
$$\nu(x,y,z)=\begin{pmatrix} a(x,y,z) & b(x,y,z) & x-a(x,y,z) \cr d(x,y,z) & e(x,y,z) & y-d(x,y,z) \cr g(x,y,z) & h(x,y,z) & z-g(x,y,z)  \end{pmatrix}.$$
Moreover, $\bullet$ is commutative if and only if
$$(x,y,z)\bullet(r,s,t)=(r,s,t)\bullet(x,y,z),\quad \forall (x,y,z), (r,s,t)\in \mathbb Z^3.$$
From this condition, we obtain the following relations:
\begin{equation*}
\begin{array}{*{20}c}
 a(x,y,z)r+b(x,y,z)s+(x-a(x,y,z))t &=& a(r,s,t)x+b(r,s,t)y+(r-a(r,s,t))z  \\
 d(x,y,z)r+e(x,y,z)s+(y-d(x,y,z))t &=& d(r,s,t)x+e(r,s,t)y+(s-d(r,s,t))z  \\
 g(x,y,z)r+h(x,y,z)s+(z-g(x,y,z))t &=& g(r,s,t)x+h(r,s,t)y+(t-g(r,s,t))z  \\
\end{array} 
\end{equation*}
In particular, if $(r,s,t)=(1,0,0)$ we obtain
\begin{equation} \label{comm1}
\begin{array}{*{20}c}
 a(x,y,z)&=& a(1,0,0)x+b(1,0,0)y+(1-a(1,0,0))z  \\
 d(x,y,z)&=& d(1,0,0)x+e(1,0,0)y-d(1,0,0)z  \\
 g(x,y,z)&=& g(1,0,0)x+h(1,0,0)y-g(1,0,0)z  \\
\end{array} 
\end{equation}
On the other hand, when $(r,s,t)=(0,1,0)$ we have
\begin{equation} \label{comm2}
\begin{array}{*{20}c}
 b(x,y,z)&=& a(0,1,0)x+b(0,1,0)y-a(0,1,0)z  \\
 e(x,y,z)&=& d(0,1,0)x+e(0,1,0)y+(1-d(0,1,0))z  \\
 h(x,y,z)&=& g(0,1,0)x+h(0,1,0)y-g(0,1,0)z  \\
\end{array} 
\end{equation}
From (\ref{comm1}), when $(x,y,z)=(0,1,0)$, we get
\begin{equation*}
\begin{array}{*{20}c}
 a(0,1,0)&=& b(1,0,0)\\
 d(0,1,0)&=& e(1,0,0)\\
 g(0,1,0)&=& h(1,0,0)\\
\end{array} 
\end{equation*}
These relations allow to rewrite equations (\ref{comm1}) and (\ref{comm2}) introducing the parameters
\begin{equation*}
\begin{array}{*{20}c}
\alpha=a(1,0,0)&\beta=a(0,1,0)=b(1,0,0)& \phi=b(0,1,0)\\
\delta=d(1,0,0)&\gamma=d(0,1,0)=e(1,0,0)&\rho=e(0,1,0)\\
\sigma=g(1,0,0)&\theta=g(0,1,0)=h(1,0,0)&\lambda=h(0,1,0)\\
\end{array}
\end{equation*}
and to finally find
\begin{equation*} \begin{pmatrix} \alpha (x - z) + \beta y + z & \beta (x - z) + \phi y & (1 - \alpha )(x - z) - \beta y \cr \delta (x - z) + \gamma y & \gamma (x - z) + \rho y + z &  - \delta (x - z) + (1 - \gamma )y \cr \sigma (x - z) + \theta y & \theta (x - z) + \lambda y & - \sigma (x - z) - \theta y + z \end{pmatrix}
\end{equation*}
Elements of $\nu(x,y,z)$ are all linear functions of $(x,y,z)$ depending on the introduced parameters. This fact is very useful in order to verify also condition 1 in Definition \ref{natural}. Indeed, we may only consider such a condition for triples $(1,0,0)$, $(0,1,0)$, and $(0,0,1)$. A (not so) little bit of calculation shows that parameters must satisfy the following system
\begin{equation*}
\begin{cases}
   \alpha \gamma  + (\rho  + \theta  - \beta )\delta  + (1 - \gamma )\sigma  = \gamma ^2 \cr
    - (\phi  - \lambda )\alpha  + \phi \gamma  - \lambda = \beta (\rho  + \theta  - \beta ) \cr
    - \sigma (\phi  - \lambda ) + \lambda \gamma  = \theta (\rho  + \theta  - \beta ) \cr
   \alpha \theta  + \delta \lambda  - \beta \sigma  = \theta \gamma  \cr
   \delta (\phi  - \lambda ) = \beta \gamma  - \theta \gamma  + \theta 
\end{cases}
\end{equation*}
We consider two cases: $\phi=\lambda$ and $\phi\neq \lambda$. 
When $\phi=\lambda$, we have
\begin{equation*}
\begin{cases}
   \alpha \gamma  + (\rho + \theta - \beta )\delta  + (1 - \gamma )\sigma  = \gamma ^2 \cr
   \phi \gamma  - \lambda  = \beta (\rho +\theta - \beta ) \cr
   \phi \gamma  = \theta(\rho+\theta-\beta) \cr
   \alpha\theta+\delta \phi  - \beta \sigma  = \theta\gamma\cr
   \beta \gamma=\theta(\gamma-1)
\end{cases}
\end{equation*}
In this case, when $\gamma\neq0$ we find the solution $\alpha=\frac{\gamma^3+\gamma^2\sigma-\gamma\sigma-\delta(\rho\gamma+\theta)}{\gamma^2}$, $\beta=\frac{\theta(\gamma-1)}{\gamma}$, $\phi=\frac{\theta(\rho\gamma+\theta)}{\gamma^2}$, corresponding to matrices
\tiny
\begin{equation*}
\mathcal A=\left[ {\begin{array}{*{20}c}
   {\frac{{\gamma ^3  + \sigma \gamma ^2  - \sigma \gamma  - \delta (\rho \gamma  + \theta )}}{{\gamma ^2 }}(x - z) + \frac{{\theta (\gamma -1 )}}{\gamma }y + z} & {\frac{{\theta (\gamma - 1 )}}{\gamma }(x - z) + \frac{{\theta (\rho \gamma  + \theta )}}{{\gamma ^2 }}y} & {\left( { \frac{{\gamma^2-\gamma ^3  - \sigma \gamma ^2  + \sigma \gamma  + \delta (\rho \gamma  + \theta )}}{{\gamma ^2 }}} \right)(x - z) - \frac{{\theta (\gamma-1 )}}{\gamma }y}  \\
   {\delta (x - z) + \gamma y} & {\gamma (x - z) + \rho y + z} & { - \delta (x - z) + (1 - \gamma )y}  \\
   {\sigma (x - z) + \theta y} & {\theta (x - z) + \frac{{\theta (\rho \gamma  + \theta )}}{{\gamma ^2 }}y} & { - \sigma (x - z) - \theta y + z}  \\
\end{array}} \right]
\end{equation*}
\normalsize

\noindent On the other hand, if $\gamma=0$ we find solution $\theta=\gamma=0$, $\phi=\lambda=-\beta(\rho-\beta)$, $\sigma=-\delta(\rho-\beta)$ and corresponding matrices are
\begin{equation*}
\mathcal B=\left[ {\begin{array}{*{20}c}
   {\alpha (x - z) + \beta y + z} & {\beta (x - z) - \beta (\rho  - \beta )y} & {(1 - \alpha )(x - z) - \beta y}  \\
   {\delta (x - z)} & {\rho y + z} & { - \delta (x - z) + y}  \\
   { - \delta (\rho  - \beta )(x - z)} & { - \beta (\rho  - \beta )y} & {  \delta (\rho  - \beta )(x - z) + z}  \\
\end{array}} \right]
\end{equation*}
When $\phi\neq\lambda$ solution is $\alpha=\frac{\phi\gamma-\lambda-\beta(\rho+\theta-\beta)}{\phi-\lambda}$, $\delta=\frac{\beta\gamma-\theta\gamma+\theta}{\phi-\lambda}$, $\sigma=\frac{\lambda\gamma-\theta(\rho+\theta-\beta)}{\phi-\lambda}$, giving matrices
\footnotesize
\begin{equation*}
\mathcal C=\left[ {\begin{array}{*{20}c}
   {\frac{{\phi \gamma  - \lambda  - \beta (\rho+ \theta - \beta )}}{{\phi  - \lambda }}(x - z) + \beta y + z} & {\beta (x - z) + \phi y} & {\left( {\frac{{\phi (1 - \gamma ) + \beta (\rho +\theta - \beta )}}{{\phi  - \lambda }}} \right)(x - z) - \beta y}  \\
   {\frac{{\beta \gamma-\theta\gamma+\theta }}{{\phi  - \lambda }}(x - z) + \gamma y} & {\gamma (x - z) + \rho y + z} & { - \frac{{\beta \gamma -\theta\gamma+\theta}}{{\phi  - \lambda }}(x - z) + (1 - \gamma )y}  \\
   {\frac{{\lambda \gamma-\theta(\rho+\theta-\beta) }}{{\phi  - \lambda }}(x - z)+\theta y} & {\theta(x-z)+\lambda y} & { - \frac{{\lambda \gamma-\theta(\rho+\theta-\beta) }}{{\phi  - \lambda }}(x - z)-\theta y + z}  \\
\end{array}} \right]
\end{equation*}
\normalsize

\section{Matrices yielding commutative monoids of Pythagorean triples}\label{sec:mono-pyt}
In the previous section, we have found all matrices $\nu(x,y,z)$ such that injection $\nu$ is natural and consequently $(\mathbb Z^3,\bullet)$ is a commutative monoid. Now we want to show that a particular subset of PTPMs lies in one of families $\mathcal A$,$\mathcal B$ and $\mathcal C$. 

Let us consider the set $\mathcal{P}=\left\lbrace(x,y,z)\in \mathbb{Z}^3: x^2+y^2=z^2\right\rbrace$ and matrices $\nu(x,y,z)$, where $(x,y,z)\in \mathcal{P}$.
We recall that generic form of $\nu(x,y,z)$, with $(x,y,z)\in \mathcal{P}$, belonging to one of families $\mathcal A$, $\mathcal B$, $\mathcal C$, is 
\begin{footnotesize}
\begin{equation} \label{standard form}
\nu(x,y,z)=\begin{pmatrix} A_1(x - z)+B_1y + z & A_2(x - z)+B_2y & A_3(x - z) + B_3y \cr A_4(x - z) +B_4 y & A_5(x-z)+B_5y+z & A_6(x - z)+B_6y \cr A_7 (x - z) + B_7y & A_8(x - z)+B_8y & A_9(x - z)+B_9y + z \end{pmatrix}.
\end{equation}
\end{footnotesize}
\begin{lemma}
Necessary conditions for (\ref{standard form}) to be a PTPM are 
\begin{equation}\label{neccond}
B_1=B_5=B_9, \quad B_7=B_3, \quad B_8=B_6, \quad B_4=-B_2.
\end{equation}
\end{lemma}
\begin{proof}
Using the well--known parametrization $x=m^2-n^2$, $y=2mn$ and $z=m^2+n^2$, matrix (\ref{standard form}) becomes
\begin{footnotesize}
\begin{equation*}
\begin{pmatrix} -2A_1n^2+2B_1mn + m^2+n^2 & -2A_2n^2+2B_2mn & -2A_3n^2 + 2B_3mn \cr -2A_4n^2+2B_4mn & -2A_5n^2+2B_5mn+m^2+n^2 & -2A_6n^2+2B_6mn \cr -2A_7n^2 + 2B_7mn & -2A_8n^2+2B_8mn & -2A_9n^2+2B_9mn + m^2+n^2  \end{pmatrix}.
\end{equation*}
\end{footnotesize}
We find, with a little bit of calculation, that condition 
$$\nu(m^2-n^2,2mn,m^2+n^2).(u^2-v^2,2uv,u^2+v^2)\in \mathcal{P}$$ 
is valid for all Pythagorean triples $(m^2-n^2,2mn,m^2+n^2)$ and $(u^2-v^2,2uv,u^2+v^2)$ if the following equality holds
\begin{footnotesize}
\begin{equation}\label{cond}
\begin{split}
[M_1^2+M_2^2-M_3^2-M_1(u^2-v^2)-2M_2uv+M_3(u^2+v^2)]n^4+\\
+[N_1(u^2-v^2)+2N_2uv-N_3(u^2+v^2)-2M_1N_1-2M_2N_2+2M_3N_3]n^3m+\\
+[N_1^2+N_2^2-N_3^2-M_1(u^2-v^2)-2M_2uv+M_3(u^2+v^2)]m^2n^2+\\
[N_1(u^2-v^2)+2N_2uv-N_3(u^2+v^2)]nm^3=0
\end{split}
\end{equation}
\end{footnotesize}
where, for $j=0,1,2$, we have
\begin{equation}\label{MN}\begin{cases}
M_{j+1}=A_{3j+1}(u^2-v^2)+2A_{3j+2}uv+A_{3j+3}(u^2+v^2)\cr N_{j+1}=B_{3j+1}(u^2-v^2)+2B_{3j+2}uv+B_{3j+3}(u^2+v^2) \end{cases}.\end{equation}
Equality \eqref{cond} is independent by the choice of involved triples. Thus, all coefficients within square brackets must be equal to 0. In particular, if we consider equality
\begin{equation}\label{N}
N_1(u^2-v^2)+2N_2uv-N_3(u^2+v^2)=0
\end{equation}
and we substitute the second relations of (\ref{MN}) into (\ref{N}) we find
\begin{equation*}
\begin{split}
(B_1+B_3-B_7-B_9)u^4+2(B_2+B_4+2B_6-2B_8)u^3v-2(B_1-2B_5+B_9)u^2v^2+\\-2(B_2+B_4-B_6+B_8)uv^3+(B_1-B_3+B_7-B_9)v^4=0
\end{split}
\end{equation*}
Since this equality must hold for all $u$ and $v$, coefficients within round brackets must be 0, leading to the system of equations
\begin{equation*}
\begin{cases}
B_1+B_3-B_7-B_9=0\\
B_2+B_4+2B_6-2B_8=0\\
B_1-2B_5+B_9=0\\
B_2+B_4-B_6+B_8=0\\
B_1-B_3+B_7-B_9=0.
\end{cases}
\end{equation*}
which easily gives necessary conditions (\ref{neccond}).
\end{proof}
Now, using previous lemma, we prove that PTPMs belong to family $\mathcal C$.

\begin{theorem}
The PTPMs $\nu(x,y,z)=M_{\beta,\gamma}(x,y,z)$, with $(x,y,z)\in \mathcal{P}$, are subset of the family $\mathcal C$ and they have the following form
\begin{tiny}
\begin{equation}\label{pythagorean}
\begin{pmatrix} (\gamma^2-\gamma +1-\beta^2)(x - z)+\beta y + z & \beta (x - z) - \gamma y & (\beta^2-\gamma^2+\gamma) (x - z)-\beta y \cr -\beta(2\gamma-1)(x - z) +\gamma y & \gamma (x-z)+\beta y+z & \beta(2\gamma-1)(x - z)+(1-\gamma)y \cr (\beta^2+\gamma^2-\gamma) (x - z)-\beta y & -\beta(x - z)+(1-\gamma)y & -( \beta^2+\gamma^2-\gamma)(x - z)+\beta y + z \end{pmatrix}
\end{equation}
\end{tiny}
\end{theorem}
\begin{proof}
First we show that conditions (\ref{neccond}) are satisfied only by a subset of $\mathcal C$. Let us consider a matrix of the form $\mathcal A$. Conditions (\ref{neccond}) correspond to relations
$$\frac{\theta(\gamma-1)}{\gamma}=\rho=-\theta, \quad \theta=-\frac{\theta(\gamma-1)}{\gamma^2}, \quad \frac{\theta(\rho\gamma+\theta)}{\gamma^2}=1-\gamma, \quad \gamma=\frac{\theta(\rho\gamma+\theta)}{\gamma^2},$$
which are clearly inconsistent.

If we examine a matrix belonging to family $\mathcal B$, we have that conditions (\ref{neccond}) lead to equalities
$$\beta=\rho=0,\quad \beta=0, \quad -\beta(\rho-\beta)=1,\quad -\beta(\rho-\beta)=0,$$
and also in this case they are clearly inconsistent.

Finally, for matrices belonging to family $\mathcal C$, conditions (\ref{neccond}) are equivalent to 
$$\beta=\rho=-\theta,\quad \theta=-\beta, \quad \lambda=1-\gamma, \quad \phi=-\gamma.$$
They are compatible, providing matrices of the form (\ref{pythagorean}).
We can prove that these matrices are PTPMs considering two triples generated by $m,n$ and $u,v$.

We have
$$(m^2-n^2,2mn,m^2+n^2)\bullet(u^2-v^2,2uv,u^2+v^2)=(A,B,C)$$
where
$$A=(mu+(1-2\gamma)nv)^2-(mv+n(u+2\beta v))^2,$$
$$B=2(mu+(1-2\gamma)nv)(mv+n(u+2\beta v)),$$
$$C=(mu+(1-2\gamma)nv)^2+(mv+n(u+2\beta v))^2,$$
which is clearly a Pythagorean triple.

On the other hand 
$$(u^2-v^2,2uv,u^2+v^2). \nu(m^2-n^2,2mn,m^2+n^2)=(A',B',C')$$
where
$$A'=(m u + n v)^2 - (m v + n (u - 2 \gamma u + 2 \beta v))^2,$$
$$ B'=2(m u + n v) (m v + n (u - 2 \gamma u + 2 \beta v)),$$
$$ C'=(m u + n v)^2 + (m v + n (u - 2 \gamma u + 2 \beta v))^2,$$
which is clearly another Pythagorean triple.
\end{proof}

The previous theorem characterizes matrices $M_{\beta,\gamma}$, depending on two parameters $\beta$ and $\gamma$, such that 
$$\nu:\mathcal P\rightarrow M_3,\quad \nu(x,y,z)=M_{\beta,\gamma}(x,y,z)$$
is natural, i.e., said $\bullet_{\beta,\gamma}$ the induced product, $(\mathcal P,\bullet_{\beta,\gamma})$ is a commutative monoid.

\section{Products over conics and groups of Pythagorean triples}\label{sec:group}

Since we have the classical parametrization for the Pythagorean triples $(x,y,z)\in \mathcal{P}$ given by
$$\phi:\mathbb Z^2\rightarrow \mathcal P,\quad \phi(u,v)=(u^2-v^2,2uv,u^2+v^2),$$
matrices $M_{\beta,\gamma}(x,y,z)$ can be rewritten as
\begin{tiny}
$$M_{\beta,\gamma}(u,v)=\begin{pmatrix} u^2+2\beta uv+(2\beta^2-2\gamma^2+2\gamma-1)v^2 & -2v(\gamma u+\beta v) & 2(-\beta^2v-\beta u+\gamma(\gamma-1))v \cr -2(\beta v-\gamma(u+2\beta v))v & u^2+2\beta uv+(-2\gamma+1)v^2 & 2(-\gamma u+u+\beta v-2\beta\gamma v)v \cr -2(\beta^2v+\beta u+\gamma(\gamma-1)v)v & 2(-\gamma u+u+\beta v)v & u^2+2\beta uv+(2\beta^2+2\gamma^2-2\gamma+1)v^2 \end{pmatrix}.$$
\end{tiny}

Now, we define a product between points of $\mathbb{R}^2$ such that it is well--defined with respect to multiplication of matrices $M_{\beta,\gamma}(u,v)$.
\begin{definition} For any $(u,v),(s,t)\in\mathbb R^2$, we define
\begin{equation} (u,v)*_{\beta,\gamma}(s,t)=(su+tv(1-2\gamma),tu+sv+2\beta tv). \end{equation}
\end{definition}
The reader can easily prove that
$$M_{\beta,\gamma}(u,v).M_{\beta,\gamma}(s,t)=M_{\beta,\gamma}((u,v)*_{\beta,\gamma}(s,t)).$$

The determinant of $M_{\beta,\gamma}(u,v)$ is surprisingly $(u^2+2\beta uv-(1-2\gamma)v^2)^3$. Thus, function
$$\delta(M_{\beta,\gamma}(u,v))=u^2+2\beta uv-(1-2\gamma)v^2$$
is multiplicative, i.e.,
$$\delta(M_{\beta,\gamma}(u,v).M_{\beta,\gamma}(u',v'))=\delta(M_{\beta,\gamma}(u,v))\cdot\delta(M_{\beta,\gamma}(u',v')),$$
since it is the cubic root of the determinant. 

In this way, study of commutative monoids of Pythagorean triples offers an unexpected connection with conics. Indeed, it is now natural to consider conics
$$\alpha_{\beta,\gamma}(z)=\{(x,y)\in\mathbb R^2: x^2+2\beta xy-(1-2\gamma)y^2=z\}.$$
From previous notation and remarks, it follows that
\begin{equation}\label{prodwz} (u,v)*_{\beta,\gamma}(s,t)\in\alpha_{\beta,\gamma}(wz),\quad \forall (u,v)\in\alpha_{\beta,\gamma}(w), (s,t)\in\alpha_{\beta,\gamma}(z). \end{equation}
\begin{proposition}
Let $(u,v)^{n_{*_{\beta,\gamma}}}$ be the $n$--th power of $(u,v)$ with respect to $*_{\beta,\gamma}$, then
$$(M_{\beta,\gamma}(u,v))^n=M_{\beta,\gamma}((u,v)^{n_{*_{\beta,\gamma}}}).$$
Consequently, power of matrices of the form \eqref{pythagorean} are still of this form.
\end{proposition}
\begin{proof}
The proof directly follows from previous observations.
\end{proof}
Conic $\alpha_{\beta,\gamma}(1)$ is especially interesting, due to equation \eqref{prodwz}. In \cite{bcm}, the authors deeply studied the conic
\begin{equation} \label{conic-hd} x^2+hxy-dy^2=1, \end{equation}
which is a group with the operation 
\begin{equation} \label{prod-hd} (u,v)\cdot(s,t)=(tu+svd,tu+sv+tvh). \end{equation}
The identity is the point $(1,0)$ and the inverse of a generic point $(x,y)$ is $(x+hy,-y)$. This conic and this product clearly coincide with $\alpha_{\beta,\gamma}(1)$ and $*_{\beta,\gamma}$ for $h=2\beta$ and $d=1-2\gamma$. Thus, $(\alpha_{\beta,\gamma}(1),*_{\beta,\gamma})$ is a group.
\begin{theorem}
If $\beta^2-2\gamma+1$ is a positive square--free integer, then we have $(u_1-\beta v_1,v_1)\in\alpha_{\beta,\gamma}(1)$, where $(u_1,v_1)$ is the minimal solution of the Pell equation
$$x^2-(\beta^2-2\gamma+1)y^2=1$$
\end{theorem}
\begin{proof}
The proof directly follows from the equality
$$x^2+2\beta xy-(1-2\gamma)y^2=(x+\beta y)^2-(\beta^2-2\gamma+1)y^2.$$
\end{proof}
As a consequence, the set 
$$N_{\beta,\gamma}=\{(u,v)\in\alpha_{\beta,\gamma}(1): u,v\in\mathbb Z\}$$ 
includes infinite points and $(N_{\beta,\gamma},*_{\beta,\gamma})$ is a group.

Let us consider $\bar \phi=\phi|_{\alpha_{\beta,\gamma}(1)}$, then $\bar \phi$ is a morphism, i.e.,
$$\bar\phi((u,v)*_{\beta,\gamma}(s,t))=\bar\phi(u,v)\bullet_{\beta,\gamma}\bar\phi(s,t),\quad \forall (u,v),(s,t)\in\alpha_{\beta,\gamma}(1).$$
Moreover, $\ker \bar \phi =\{(\pm1,0)\}$ and
$$(Im(\bar\phi),\bullet_{\beta,\gamma})\cong (N_{\beta,\gamma})/\{(\pm1,0)\}.$$
Thus we have the following theorem.
\begin{theorem} \label{pyt-group}
$(Im(\bar\phi),\bullet_{\beta,\gamma})$ is a group of Pythagorean triples and the inverse of the Pythagorean triple $\bar \phi(u,v)$ is 
$$\bar \phi(u+2\beta v, -v)=(-v^2+(u+2\beta v)^2,-2v(u+2\beta v), v^2(u+2\beta v)^2).$$
\end{theorem}
\begin{example}
Let us consider the Pythagorean triple $(3,4,5)$. Since we have $\phi(2,1)=(3,4,5)$, we are interested in conics $\alpha_{\beta,\gamma}(1)$ containing the point $(2,1)$. It is easy to see that $(2,1)\in\alpha_{\beta,-2\beta-1}(1)$, $\forall \beta \in\mathbb R$. However, we only consider $\beta \in\mathbb Z$. By Theorem \eqref{pyt-group}, inverse of $(2,1)$ is $(2(\beta+1),-1)$ and inverse of $(3,4,5)$ is 
$$\bar\phi(2(\beta+1),-1)=(4\beta^2+8\beta+3,-4(\beta+1),4(\beta+1)^2+1).$$

Let us set, e.g., $\beta=1$, the conic $\alpha_{1,-3}(1)$ is
$$x^2+2xy-7y^2=1.$$
The inverse of $(2,1)$ is $(4,-1)$ and inverse of $(3,4,5)$, with respect to $\bullet_{1,-3}$, is $(15,-8,17)$. Indeed,
$$(3,4,5)\bullet_{1,-3}(15,-8,17)=M_{1,-3}(3,4,5).(15,-8,17)=$$
$$=\begin{pmatrix}  -15 & 10 & 18 \cr -26 & 15 & 30 \cr -30 & 18 & 35 \end{pmatrix}\begin{pmatrix} 15 \cr -8 \cr 17 \end{pmatrix}=\begin{pmatrix} 1 \cr 0 \cr 1 \end{pmatrix}.$$
Moreover, we can evaluate powers of Pythagorean triples:
$$(3,4,5)^{2_{\bullet_{1,-3}}}=\begin{pmatrix}  -15 & 10 & 18 \cr -26 & 15 & 30 \cr -30 & 18 & 35 \end{pmatrix}\begin{pmatrix} 3 \cr 4 \cr 5 \end{pmatrix}=\begin{pmatrix} 85 \cr 132 \cr 157 \end{pmatrix}.$$
The Pythagorean triple $(85,132,157)$ corresponds to the point $(2,1)^{2{*_{1,-3}}}=(11,6)\in\alpha_{1,-3}(1)$. Thus, the inverse of $(11,6)$ is $(23,-6)$ and the inverse of $(85,132,157)$ is $\phi(23,-6)=(493,-276,565)$.
\end{example}

\section{Conclusion}\label{sec:conc}
The paper provides new structures of commutative monoids and commutative groups over sets of Pythagorean triples, starting from \emph{natural} injiections between triples and $3 \times 3$ matrices. Moreover, this study has connection with groups over conics. 

As future developments, we find interesting to study operations $*_{\beta,\gamma}$ and $\bullet_{\beta,\gamma}$ over finite fields $\mathbb Z_p$ for $p$ prime. Moreover, this study offers the possibility to use Pythagorean triples in the approximation of quadratic irrationalities. Indeed, in \cite{bcm}, conic \eqref{conic-hd} and product \eqref{prod-hd} have been profitably used in Diophantine approximation for quadratic irrationalities. Thus, the above connections with Pythagorean triples show that they can be used in this context.


\begin{thebibliography}{99}

\bibitem{bcm} S. Barbero, U. Cerruti, N. Murru, \emph{Generalized Rédei rational functions and rational approximations over conics}, International Journal of Pure and Applied Mathematics, Vol. \textbf{64}, No. \textbf{2}, 305--317, 2010.

\bibitem{Bea} R. A. Beauregard, E. R. Suryanarayan, \emph{Pythagorean triples: the hyperbolic view}, The College Mathematics Journal, Vol. \textbf{27}, No. \textbf{3}, 170--181, 1996.

\bibitem{Cra} M. Crasmareanu, \emph{A new method to obtain Pythagorean triple preserving matrices}, Missouri Journal of Mathematical Sciences, Vol. \textbf{14}, No. \textbf{3}, 2002.

\bibitem{Eck} E. Eckert, \emph{The group of primitive Pythagorean triangles}, Mathematics Magazine, Vol. \textbf{54}, 22--27, 1984.

\bibitem{Pal1} L. Palmer, M. Ahuja, M. Tikoo, \emph{Finding Pythagorean triple preserving matrices}, Missouri Journal of Mathematical Sciences, Vol. \textbf{10}, 99--105, 1998.

\bibitem{Pal2} L. Palmer, M. Ahuja, M. Tikoo, \emph{Constructing Pythagorean triple preserving matrices}, Missouri Journal of Mathematical Sciences, Vol. \textbf{10}, 159--168, 1998.

\bibitem{Som} C. Somboonkulavudi, A. Harnchoowong, \emph{Pythagorean triples over Gaussian integers}, International Journal of Algebra, Vol. \textbf{6}, No. \textbf{2}, 55--64, 2012.

\bibitem{Tik} M. Tikoo, \emph{A note on Pythagorean triple preserving matrices}, Int. J. Math. Educ. Sci. Technol., Vol. \textbf{33}, No. \textbf{6}, 893--894, 2002.

\bibitem{Tau} O. Taussky, \emph{Sum of squares}, American Mathematical Monthly, Vol. \textbf{77}, 805--830, 1970.

\bibitem{Zan} P. Zanardo, U. Zannier, \emph{The group of Pythagorean triples in number fields}, Annali di Matematica Pura ed Applicata, Vol. \textbf{CLIX}, 81--88, 1991.
  
\end{thebibliography}
\end{document}